\documentclass[10pt,oneside,a4paper]{amsart}
\usepackage{latexsym}
\usepackage{amsfonts,amsmath,amssymb,indentfirst}
\usepackage[english]{babel}
\usepackage[all]{xy}
\usepackage[pdftex]{hyperref}
\newcommand{\bdism}{\begin{displaymath}}
\newcommand{\edism}{\end{displaymath}}

\newcommand{\rr}{\mathbb{R}}
\newcommand{\qq}{\mathbb{Q}}

\newcommand{\nn}{\mathbb{N}}
\newcommand{\pp}{\mathbb{P}}
\newcommand{\oo}{\mathcal{O}}

\newcommand{\T}{\mathcal{T}}
\newcommand{\D}{\mathcal{D}}
\DeclareMathOperator{\lct}{lct}
\DeclareMathOperator{\LCT}{LCT}
\DeclareMathOperator{\vol}{vol}

\newtheorem{theorem}{Theorem}[section]

\newtheorem{corollary}[theorem]{Corollary}
\newtheorem{lemma}[theorem]{Lemma}

\newtheorem{definition}[theorem]{Definition}
\newtheorem{example}[theorem]{Example}
\newtheorem{question}[theorem]{Question}

\address{Department of Mathematics, Princeton University, Princeton NJ 08544-1000,
USA} \email{gdi@math.princeton.edu}

\author{\scshape Gabriele Di Cerbo}
\title{\bf On Fujita's log spectrum conjecture}
\begin{document}
\pagestyle{headings}
\begin{abstract}
We prove Fujita's log spectrum conjecture. It follows from the ACC of a suitable set 
of pseudo-effective thresholds.
\end{abstract}
\maketitle

\tableofcontents

\section{Introduction}
\pagenumbering{arabic}

In a series of papers, Fujita introduced what he called the Kodaira energy of a 
polarized variety and he used it to study the adjunction theory of polarized varieties. See \cite{Fujita1}, \cite{Fujita2}, \cite{Fujita} and \cite{Fujita3} for details.
Let us recall the definition.

\begin{definition}
Let $(X,\Delta)$ a lc pair and let $M$ be a big divisor. The Kodaira energy of $(X,\Delta,M)$ is defined by 
\bdism
\kappa\epsilon(X,\Delta,M):=-\inf\left\{ t \in \qq_{\geq 0} \:|\: \kappa(K_{X}+\Delta+t M)\geq 0\right\}.
\edism
\end{definition}
 
Of course the interesting case is when $K_{X}+\Delta$ is not pseudo-effective. 
Using techniques from the minimal model program, Fujita proved that $\kappa\epsilon(X,\Delta,M)$ is a rational 
number if $(X,\Delta)$ is log terminal, $X$ has dimension at most three and $M$ is big and nef, see Theorem 2.2 in \cite{Fujita}.
Moreover in \cite{Fujita1}, he stated the so called spectrum conjecture. It says that if $\Delta=0$ and $M$ is ample then 
the set of negative Kodaira energies is finite away from zero. The goal of this work is to prove 
an analogue of the spectrum conjecture for pairs. It is easy to see that in this case the set of Kodaira energies 
is not discrete. On the other hand, Fujita conjectured that the limit points can be reached only from below, see \cite{Fujita} and \cite{Fujita3}. This is the content of the log spectrum conjecture and it is one of the main results of this paper.

\begin{theorem}\label{spectrum}
Let $S^{ls}_{n}$ be the set of Kodaira energies where $(X,\Delta)$ is a pair with $X$ smooth, $\Delta$ is a reduced divisor with simple normal crossing support and $M$ is an ample Cartier divisor on $X$. 
Then for any real number $x<0$ the set $\left\{t\in S^{ls}_{n}\: | \: x<t<x+\epsilon \right\}$ is finite, for $\epsilon$ small enough.
\end{theorem}

Fujita proved that if $\Delta=\emptyset$ and $\dim(X)\leq 3$ then the set of negative Kodaira energies is finite away from zero using 
boundedness of Fano threefolds with terminal singularities. See Theorem 3.4.1 in \cite {Fujita}. Our approach is different. We will show how techniques from the minimal model program and the recent results in \cite{HMX} allow us to prove Theorem \ref{spectrum}. We will actually prove a more general result on a set of pseudo-effective thresholds. Moreover, this will imply that the Kodaira energy is a rational number in many cases. Note that the question of the rationality of the Kodaira energy has been extensively studied. For example see \cite{Alexeev}, \cite{Araujo}, \cite{BCHM} and \cite{Demailly}.

Throughout this paper we will work with the pseudo-effective threshold instead of the Kodaira energy. It seems that the former invariant is more suitable for the minimal model program and it applies to a more general setting. Let us recall its definition. 

\begin{definition}
Let $(X,\Delta)$ a lc pair and $M$ an effective divisor such that $K_{X}+\Delta+tM$ is pseudo-effective for some $t\geq 0$. The pseudo-effective threshold $\tau(X,\Delta;M)$ of $M$ with respect to $(X,\Delta)$ is defined by
\bdism
\tau(X,\Delta;M):=\inf\left\{ t \in \rr_{\geq 0} \:|\: K_{X}+\Delta+t M \:\text{is pseudo-effective}\right\}.
\edism
\end{definition}

Note that in the case $L$ is big, the pseudo-effective threshold and the Kodaira energy are easily related by the formula
\bdism
\kappa\epsilon(X,\Delta,L)=-\tau(X,\Delta;L).
\edism

We will study an analogue of Theorem \ref{spectrum} for pseudo-effective thresholds in a slightly more general form, namely we allow the coefficients of $\Delta$ and $M$ to vary in some fixed sets. In order to simplify the presentation of our results, let us first introduce some notation.

\begin{definition}
Fix a positive integer $n$ and two subsets $I\subseteq \left[0,1\right]$ and $J\subseteq \rr$. Then we define $\D_{n}(I,J)$ to be the set of all lc pairs $(X,\Delta)$ with an effective $\rr$-divisor $M$ on $X$ such that $X$ is a projective variety of dimension $n$, the coefficients of $\Delta$ (resp. $M$) are in $I$ (resp. $J$) with $(X,\Delta+t M)$ lc and $K_{X}+\Delta+ t M$ pseudo-effective for some $t\geq 0$.
\end{definition} 

The main object of study of this paper is the following set.

\begin{definition}
Fix a positive integer $n$ and two subsets $I\subset [0,1]$ and $J\subseteq \rr$. We define the set 
\bdism
\T_{n}(I,J):=\left\{\tau(X,\Delta;M) \:|\: (X,\Delta;M)\in\D_{n}(I,J)\right\}.
\edism
\end{definition}

We can now state our main theorem.

\begin{theorem}\label{theo}
Fix a positive integer $n$ and two sets $I\subseteq [0,1]$ and $J\subseteq \rr$. If $I$ and $J$ satisfy the DCC then
$\T_{n}(I,J)$ satisfies the ACC. 
\end{theorem}

Note that we don't require $M$ to be an ample divisor. If the coefficients of $\Delta$ or $M$ are not in a DCC set, then it is easy to see that the above result is false even in dimension one. 

Theorem \ref{theo} implies that the Kodaira energy is a rational number if $(X,\Delta,M)$ satisfy the above conditions,
see Corollary \ref{rational}. 

Our main motivation in studying the ACC for pseudo-effective thresholds comes from effective birationality results of certain adjoint linear systems. See Section 4 for details. On the other hand we believe that these results have other applications. For example the pseudo-effective threshold is fundamental in the construction of Tamagawa numbers and it appears in formulas computing the asymptotic numbers of rational points of bounded heights. See \cite{Batyrev}, \cite{Fujita2} and \cite{Yuri} for details.

The results of this paper are based on the recent work of Hacon, M$^{\text{c}}$Kernan and Xu \cite{HMX}. 
Furthermore some of the main ideas are inspired by Section 8 in \cite{Demailly} and \cite{Gongyo}.

\section{Preliminaries}

Let us define what ACC and DCC stand for.

\begin{definition}
Let $I\subseteq \rr$. 

We say that $I$ satisfies the ACC (ascending chain condition) if every non-decreasing sequence $\left\{x_{i}\right\}\subset I$ is eventually constant. 

We say that $I$ satisfies the DCC (descending chain condition) if every non-increasing sequence $\left\{x_{i}\right\}\subset I$ is eventually constant.
\end{definition}

As stated in the introduction we will work with singular pairs. 

\begin{definition}
A pair $(X,\Delta)$ consists of a normal variety $X$ and a
$\rr$-Weil divisor $\Delta\geq 0$ such that $K_{X}+\Delta$ is
$\rr$-Cartier.
\end{definition}

We will mainly deal with dlt and lc pairs. For the definitions and some properties of these singularities 
we refer to \cite{KM}.

Given any effective divisor $M$ on a variety $X$, we define a global invariant which measures the singularities of $M$ and $X$.

\begin{definition}
Let $(X,\Delta)$ be a lc pair and let $M$ be an effective $\rr$-dvisor. 
The log canonical threshold of $M$ with respect to $(X,\Delta)$ is defined by
\bdism
\lct(X,\Delta;M):=\sup\left\{ t \in \rr_{\geq 0} \:|\: (X,\Delta+t M) \:\text{is lc}\right\}.
\edism
Fix a positive integer $n$ and two subsets $I,J\subseteq \rr$. Let
\bdism
\LCT_{n}(I,J):=\left\{\lct(X,\Delta;M)\right\},
\edism 
where $(X,\Delta)$ is a lc pair and the coefficients of $\Delta$ (resp. $M$) belong to $I$ (resp. $J$).
\end{definition}

One of the main theorems in \cite{HMX} is that the above set satisfies the ACC. 

\begin{theorem}[Hacon, M$^{\text{c}}$Kernan and Xu]\label{lct}
Fix a positive integer $n$, $I\subseteq \left[0,1\right]$ and $J\subseteq \rr$.
If $I$ and $J$ satisfy the DCC then $\LCT_{n}(I,J)$ satisfies the ACC.
\end{theorem}

The proof of the above theorem is a very intricate induction procedure based on four theorems. 
One of these theorems will be fundamental for us. See Theorem 1.5 in \cite{HMX}.

\begin{theorem}[Hacon, M$^{\text{c}}$Kernan and Xu]\label{Hacon}
Fix a positive integer $n\in\nn$ and a set $I\subseteq [0,1]$, which satisfies the DCC. Then there is a finite subset $I_{0}\subseteq I$ such that if
\begin{enumerate}
\item $X$ is a projective variety of dimension $n$,
\item $(X,\Delta)$ is log canonical,
\item $\Delta=\sum\delta_{i}\Delta_{i}$ where $\delta_{i}\in I$,
\item $K_{X}+\Delta\equiv 0$,
\end{enumerate}
then $\delta_{i}\in I_{0}$.
\end{theorem}

In order to apply Theorem \ref{Hacon} we need to run the minimal model program with scaling. 
When $K_{X}+\Delta$ is not pseudo-effective, Corollary 1.3.2 in \cite{BCHM} allows us to run the MMP. 
The following is a generalization of results in \cite{BCHM} to dlt pairs. See Theorem 2.15 in \cite{Lohmann} for details.

\begin{theorem}\label{MMP}
Let $(X,\Delta)$ be a dlt pair such that $K_{X}+\Delta$ is not pseudo-effective. Let $H$ be an ample divisor such that
$K_{X}+\Delta+H$ is nef. Then any $(K_{X}+\Delta)$-MMP with scaling of $H$ terminates with a Mori fiber space.
\end{theorem}

In what follows we will need dlt models. We state here the result we need and we 
refer to \cite{Kollar} for a proof.

\begin{lemma}\label{dltmodel}
Let $(X,\Delta)$ be a lc pair. Then there is a proper birational morphism $g:X^{dlt}\rightarrow X$ 
with exceptional divisors $E_{i}$ such that 
\begin{enumerate}
\item $(X^{dlt},g^{-1}_{*}\Delta+\sum E_{i})$ is dlt,
\item $X^{dlt}$ is $\qq$-factorial,
\item $K_{X^{dlt}}+g^{-1}_{*}\Delta+\sum E_{i}\sim_{\qq}g^{*}(K_{X}+\Delta)$.
\end{enumerate}
\end{lemma}

We can now apply the above lemma to reduce the study of pseudo-effective thresholds from lc pair to dlt pairs.

\begin{lemma}\label{dlt}
Let $(X,\Delta)$ be a pair with an effective divisor such that $(X,\Delta+tM)$ is lc and $K_{X}+\Delta+ tM$ is 
pseudo-effective for some $t\geq 0$. Let $I$ and $J$ be two sets such that the coefficients of $\Delta$ (resp. $M$) are
in $I$ (resp. $J$). 
Then there exists $X'$ with effective divisors $\Delta'$ and 
$M'$ such that: 
\begin{enumerate}
\item $K_{X'}+\Delta'+ tM'$ is dlt and pseudo-effective for some $t\geq 0$,
\item the coefficients of $\Delta'$ (resp. $M$) belong to $I\cup \left\{1\right\}$ (resp. $J$),
\item $\tau(X,\Delta;M)=\tau(X',\Delta';M')$.
\end{enumerate}  
\end{lemma}

\begin{proof}
Let $\tau:=\tau(X,\Delta;M)$. By assumption $(X,\Delta+\tau M)$ is a lc pair. By Lemma \ref{dltmodel} we can take 
its dlt model $(X^{dlt},g^{-1}_{*}\Delta+\sum E_{i}+\tau g^{-1}_{*}M)$.

Let $X':=X^{dlt}$, $\Delta':=g_{*}^{-1}\Delta +\sum E_{i}$ and $M':=g^{-1}_{*} M$, then 
$K_{X'}+\Delta'+\tau M'\sim_{\qq} g^{*}(K_{X}+\Delta+\tau M)$. Since the right hand side is pseudo-effective,
we have that $\tau(X',\Delta'; M') \leq \tau$. On the other hand 
$g_{*}(K_{X'}+\Delta'+t M')\sim_{\qq}K_{X}+\Delta+t M$ for any $t\geq 0$, thus 
$\tau(X',\Delta'; M') = \tau$.  
\end{proof}

\section{ACC for pseudo-effective thresholds}

In this section we prove Theorem \ref{theo} and we give some applications. In particular we will prove Theorem \ref{spectrum}.
First we show that if $J$ satisfies the DCC, or more generally if it admits a strictly positive minimum, then there are no diverging increasing sequences in $\T_{n}(I,J)$. This is a simple consequence of the fact that $(X,\Delta+t M)$ is lc. 

\begin{lemma}
Fix a positive integer $n$ and two subsets $I\subset [0,1]$ and $J\subseteq \rr$. If $J$ admits a strictly positive minimum then there exists a positive real number $\alpha$ such that $\T_{n}(I,J)\subseteq \left[0,\alpha\right]$.
\end{lemma}

\begin{proof}
Let $\delta:=\min (J\backslash\left\{0\right\})$. Since $(X,\Delta+t M)$ is lc it implies that also 
$(X,\Delta+\tau M)$ is lc, where $\tau:=\tau(X,\Delta;M)$. Then the coefficients of $\Delta+\tau M$ are less than one and in particular $\T_{n}(I,J)\subseteq \left[0,\alpha\right]$, where $\alpha=\delta^{-1}$. 
\end{proof}

Note that in general if we don't ask that $(X,\Delta+t M)$ is lc, we can find diverging sequences of pseudo-effective thresholds even if $I=\emptyset$ and $J$ is a finite set. So without any assumption it does not satisfy the ACC.

\begin{example}
Let $X_{e}:=\pp(\oo_{C}\oplus \mathcal{L})$ where $C\cong\pp^{1}$ and $\mathcal{L}$ is a line bundle of degree $-e\leq 0$. 
$X_{e}$ is usually called the e-th Hirzebruch surface. Recall that its 
canonical divisor is given numerically by $K_{X_{e}}\equiv -2C_{0}-(e+2)f$. 
Let $M=C_{0}+f$. It is easy to see that $M$ is big but not nef. A simple computation shows that 
$\tau(X_{e},M)\geq e+2$.
\end{example}

On the other hand, if we add some positivity to the divisor $M$ we can obtain a uniform upper bound on $\tau(X,\Delta;M)$ without asking the pair $(X,\Delta+\tau M)$ to be lc. For example, a theorem of Koll\'ar in \cite{Kol} implies that if $M$ is a big and nef Cartier divisor then 
$\tau(X,\Delta;M)\leq \binom{n+1}{2}$. 

We are now ready to prove Theorem \ref{theo}. Compare with Proposition 8.7 in \cite{Demailly} and Lemma 3.1 in \cite{Gongyo}.

\begin{proof}[Proof of Theorem \ref{theo}]
Let $\tau:=\tau(X,\Delta;M)\in\T_{n}(I,J)$ such that $\tau>0$. By Lemma \ref{dlt}, we can assume that $(X,\Delta+\tau M)$ is a dlt pair. 
For any $x<\tau$ the pair $(X,\Delta+x M)$ is also dlt by Corollary 2.39 in \cite{KM}. 
By Theorem \ref{MMP}, we can run a $K_{X}+\Delta+xM$ minimal model with scaling and end with a $K_{X}+\Delta+xM$ Mori fiber space $Y_{x}\rightarrow Z_{x}$ with generic fiber $F_{x}$. Let $f:X\dashrightarrow Y_{x}$ be the birational map given by the MMP. 
By definition of Mori fiber space, we have that
\begin{equation}\notag
K_{F_{x}}+\Delta_{F_{x}}+ x M_{F_{x}}:=(K_{Y_{x}}+f_{*}(\Delta+ x M))|_{F_{x}},
\end{equation} 
is anti-ample and in particular $x<\eta(x):=\tau(F_{x},\Delta_{F_{x}};M_{F_{x}})$. Note that the coefficients of $\Delta_{F_{x}}$ (resp. $M_{F_{x}}$) belong to $I$ (resp. $J$). Since $K_{F_{x}}+\Delta_{F_{x}}+\tau M_{F_{x}}$ is pseudo-effective, $\eta(x)\leq\tau$. Then $x<\eta(x)\leq\tau$, for any $0<x\leq\tau$. 

We want to prove that if $x$ is sufficiently close to $\tau$ then $(F_{x},\Delta_{F_{x}}+\eta(x) M_{F_{x}})$ is a lc pair. 
Suppose by contradiction that $\lct(x):=\lct(F_{x},\Delta_{F_{x}};x M_{F_{x}})<\eta(x)$ for any $x$ close to $\tau$. 
We define an increasing sequence $\left\{x_{i}\right\}$ converging to $\tau$ with the following properties:
$0<x_{i}<\tau$ and $\lct(x_{i})<x_{i+1}<\eta(x_{i})$. Since $(F_{x},\Delta_{F_{x}}+x_{i}M_{F_{x}})$ is dlt, we know that $x_{i}\leq \lct(x_{i})$. In particular, we get an increasing sequence of lc thresholds converging to $\tau$. This contradict the ACC for log canonical thresholds, Theorem \ref{lct}.   

We prove the theorem by induction on the dimension of $X$. If $\dim(X)=1$ the result is obvious. 

Let $(X_{i},\Delta_{i})$ be a sequence of pairs with effective divisors $M_{i}$ on $X_{i}$ such that the sequence $\tau_{i}:=\tau(X_{i},\Delta_{i};M_{i})$ is non-decreasing. We want to prove that the sequence is eventually constant. Choose $0<x_{i}<\tau_{i}$ sufficiently close to $\tau_{i}$ such that 
the pair $(F_{i},\Delta_{F_{i}}+\eta_{i}M_{F_{i}})$ is lc, where $F_{i}:=F_{x_{i}}$ and $\eta_{i}:=\tau(F_{i},\Delta_{F_{i}};M_{F_{i}})$.
Recall that $x_{i}<\eta_{i}\leq \tau_{i}$. Furthermore, we can choose the sequence $\left\{x_{i}\right\}$ such that each $0<x_{i}<\tau_{i}$ satisfy also $x_{i}>\tau_{i-1}$. Since $\eta_{i}>x_{i}>\eta_{i-1}$, we have that the sequence $\left\{\eta_{i}\right\}$ is increasing. By construction, it is enough to prove that the sequence $\left\{\eta_{i}\right\}$ is eventually constant. 

Passing to a non-decreasing subsequence we can assume that $\dim(F_{i})$ is constant. If $\dim(F_{i})<n$ then by induction we conclude that the sequence $\left\{\eta_{i}\right\}$ is eventually constant. So we can assume that $Y_{i}$ has Picard number one for any $i$. Then $K_{Y_{i}}+\Delta_{Y_{i}}+\eta_{i}M_{_{i}}\equiv 0$ and the pair $(Y_{i},\Delta_{Y_{i}}+\eta_{i}M_{Y_{i}})$ is lc. 
Since $\left\{\eta_{i}\right\}$ is non-decreasing, the coefficients of $\Delta_{Y_{i}}+\eta_{i}M_{Y_{i}}$ are in a DCC set. 

We can now apply Theorem \ref{Hacon} to the pairs $(Y_{i},\Delta_{Y_{i}}+\eta_{i}M_{Y_{i}})$. It implies that the sequence 
$\left\{\eta_{i}\right\}$ is eventually constant and then also $\left\{\tau_{i}\right\}$ is eventually constant.
\end{proof}

As a first consequence, we prove Theorem \ref{spectrum} stated in the introduction. 

\begin{proof}[Proof of Theorem \ref{spectrum}]
The proof is by contradiction. Recall that $\kappa\epsilon(X,\Delta,M)=-\tau(X,\Delta;M)$. Assume that the statement of Theorem \ref{spectrum} does not hold, then there exists a sequence of pairs $(X_{i},\Delta_{i})$ with ample divisors $M_{i}$ such that their pseudo-effective thresholds $\tau_{i}:=\tau(X_{i},\Delta_{i};M_{i})$ define an increasing 
sequence that converges to some positive number. 
Let $H'_{i}:=K_{X_{i}}+\Delta_{i}+ 2 n^{2} M_{i}$. By Theorem 1.1 in \cite{Fujino}, we have that $H'_{i}$ is a globally generated ample divisor on $X_{i}$ and $K_{X_{i}}+\Delta_{i}+H'_{i}$ is pseudo-effective. Choose a general element $H_{i}\in |H'_{i}|$ such that $(X_{i},\Delta_{i}+H_{i})$ is lc. Let $\eta_{i}:=\tau(X_{i},\Delta_{i};H_{i})$. Note that $\eta_{i}\in \T_{n}(\left\{1\right\},\left\{1\right\})$ for every $i$.
Furthermore
\bdism
K_{X_{i}}+\Delta_{i}+ tH_{i}\sim (t+1)(K_{X_{i}}+\Delta_{i})+2n^{2}t M_{i}. 
\edism

This implies that 
\bdism
\tau_{i}=2n^{2}\left(1-\frac{1}{\eta_{i}+1}\right).
\edism

A simple computation shows that the sequence $\left\{\eta_{i}\right\}$ is increasing as well. This contradicts Theorem \ref{theo} with
$I=J=\left\{1\right\}$.

\end{proof}

Looking at the proof of the Theorem \ref{theo}, we deduce a structure result for the set $\T_{n}(I,J)$. 
To make the presentation easier, we introduce the following set. 

\begin{definition}
Define 
\bdism
\T^{o}_{n}(I,J):=\left\{\tau(X,\Delta;M)\in\T_{n}(I,J)\:|\:\text{$\rho(X)=1$}\right\},
\edism
where $\rho(X)$ is the Picard number of $X$.
\end{definition}

\begin{corollary}\label{decomposition}
Fix a positive integer $n$ and two sets $I\subseteq \left[0,1\right]$ and $J\subseteq \rr$. If $I$ and $J$ satisfy the DCC then
\bdism
\T_{n}(I,J)=\bigcup_{j=1}^{n}\T^{o}_{j}(I,J).
\edism
\end{corollary}

\begin{proof}
First we prove that $\T_{n}(I,J)\supseteq\bigcup_{j=1}^{n}\T^{o}_{j}(I,J)$. Clearly it is enough 
to prove that $\T_{n-1}(I,J)\subseteq\T_{n}(I,J)$. Let $E$ be an elliptic curve. If $(X,\Delta;M)\in \D_{n-1}(I,J)$,
we define $X'=X\times E$, $\Delta'=\Delta\times E$ and $M'=M\times E$. By construction $(X',\Delta';M')\in \D_{n}(I,J)$ 
and they have the same pseudo-effective threshold.

It remains to prove the other inclusion. If we prove that 
\bdism
\T_{n}(I,J)\subseteq\T^{o}_{n}(I,J)\cup\bigcup_{j=1}^{n-1}\T_{j}(I,J),
\edism
then we are done applying induction on $n$. 

Let $\tau:=\tau(X,\Delta;M)\in\T_{n}(I,J)$. If $\rho(X)=1$ then $\tau\in\T^{o}_{n}(I,J)$. Suppose that $\rho(X)>1$. Let $\left\{x_{i}\right\}$ be a non-decreasing sequence such that $0<x_{i}<\tau$, $\lim x_{i}=\tau$ and $x_{i}>\eta(x_{i-1})$, where $\eta(x_{i})$ is obtained as in the proof of Theorem \ref{theo}. Then $x_{i}<\eta(x_{i})\leq\tau$ for any $i$ and $\lim\eta(x_{i})=\tau$. By Theorem \ref{theo} the sequence $\left\{\eta(x_{i})\right\}$ is eventually constant and in particular, there exists an $i$ such that $\eta(x_{i})=\tau$.  
\end{proof}

Fujita in \cite{Fujita} proved a similar result when $n=3$ and $\Delta=\emptyset$. 

As a corollary of the structure theorem of $\T_{n}(I,J)$, we have that the set of pseudo-effective thresholds is contained in the rational numbers. 

\begin{corollary}
If $I\subseteq \left[0,1\right]\cap\qq$ and $J\subseteq \qq$ satisfy the DCC then 
$\T_{n}(I,J)\subseteq \qq\cap\left[0,1\right]$.
\end{corollary}

\begin{proof}
We prove the result by induction on the dimension of $X$. If $\dim(X)=1$ it is obvious. By Corollary \ref{decomposition}, we have that 
\bdism
\T_{n}(I,J)=\bigcup_{j=1}^{n}\T^{o}_{j}(I,J).
\edism
By induction, it is enough to prove that $\T^{o}_{n}(I,J)\subseteq \qq\cap\left[0,1\right]$. If $\tau\in\T^{o}_{n}(I,J)$ is different from zero,
then $M$ is an ample divisor. Since $K_{X}+\Delta +\tau M\equiv 0$, we see that $\tau$ is the the nef threshold of $M$ with respect to $(X,\Delta)$. Then the Rationality Theorem in \cite{KM} implies that $\tau\in\qq$. 
\end{proof}

In particular, if we fix a lc pair $(X,\Delta)$ we get Proposition 8.7 in \cite{Demailly}.

\begin{corollary}\label{rational}
Let $(X,\Delta)$ be a lc pair with $\Delta$ an effective $\qq$-divisor. Let $M$ be an effective $\qq$-divisor such that $(X,\Delta+t M)$ is lc and $K_{X}+\Delta+tM$ is pseudo-effective for some $t\geq 0$. Then $\tau(X,\Delta;M)$ is a rational number.   
\end{corollary}

\begin{proof}
Let $I$ (resp. $J$) be the set of coefficients of $\Delta$ (resp. $M$). Since they are both finite sets,
we can apply the above corollary.
\end{proof}

We can now give a generalization to lc pairs of Corollary 1.1.7 in \cite{BCHM}.

\begin{corollary}
Let $(X,\Delta)$ a lc pair with $\Delta$ an effective $\qq$-divisor and $M$ an ample $\qq$-Cartier divisor. Then $\tau(X,\Delta;M)$ is a rational number.
\end{corollary}

\begin{proof}
Replacing $M$ with a sufficiently high power of itself, we can assume that $\tau(X,\Delta;M)<1$. The result follows from the fact that we can find an effective divisor $M'$ such that $M'\sim_{\qq} M$ and $(X,\Delta+M')$ is lc.
\end{proof}

If $(X,\Delta)$ is a klt pair we can generalize the above statement when $M$ is only big and nef. 

\begin{corollary}
Let $(X,\Delta)$ a klt pair with $\Delta$ an effective $\qq$-divisor and $M$ a big and nef $\qq$-Cartier divisor. Then $\tau(X,\Delta;M)$ is a rational number.
\end{corollary}

\begin{proof}
As before, we can assume that $\tau(X,\Delta;M)<1$. By Proposition 2.61, we can write $M\equiv A_{k}+(1/ k)E$ for $k$ large enough. In the previous decomposition $A_{k}$ is an ample $\qq$-divisor and $E$ is an effective $\qq$-divisor. Choosing $k$ sufficiently large, we can assume that $(X,\Delta+(1/ k)E)$ is a klt pair. Finally, there exists an effective divisor $A'$ that is $\qq$-linearly equivalent to $A_{k}$ such that $(X,\Delta+A'+(1/ k)E)$ is klt. Note that $M\equiv A'+(1/ k)E$ so we don't change the pseudo-effective threshold. Then Corollary \ref{rational} applies.
\end{proof}

Let us conclude this section with a slightly different version of Theorem \ref{spectrum}. 

\begin{corollary}\label{ample}
Fix a positive integer $n$ and a DCC set $I\subseteq \left[0,1\right]$. Let $\T_{n}(I)$ be the set of pseudo-effective thresholds $\tau(X,\Delta;M)$ where $(X,\Delta)$ is a klt pair such that $\dim(X)=n$, the coefficients of $\Delta$ are in $I$, $-(K_{X}+\Delta)$ is nef and $M$ is a big and nef Cartier divisor on $X$. Then the set $\T_{n}(I)$ satisfies the ACC.
\end{corollary}

\begin{proof}
Since $M-K_{X}-\Delta$ is big and nef, Theorem 1.1 in \cite{Kol2} implies that we can find a positive integer $m'$, which depends only on $n$, such that $m'M$ is base point free. Choose $m:=\max\left\{m',\binom{n+1}{2}\right\}$. Then Theorem 5.8 in \cite{Kol} implies that $K_{X}+\Delta+mM$ is pseudo-effective. Since $mM$ is base point free, we can choose a general element $H\in |mM|$ such that the coefficients of $H$ are all equal to one, $(X,\Delta+H)$ is lc and $K_{X}+\Delta+H$ is pseudo-effective. Finally, Theorem \ref{theo} gives the result. 
\end{proof}

\section{Effective birationality}

The following result tells us that the pseudo-effective threshold has a key role in questions of effective birationality of certain adjoint linear systems. See Proposition 3.2 in \cite{Gabriele}. 

\begin{theorem}\label{effective}
Let $(X,\Delta)$ be a klt pair. Let $M$ be a big and nef Cartier divisor on $X$ such that $K_{X}+\Delta+M$ is big. Let $\tau:=\tau(X,\Delta;M)$ be the pseudo-effective threshold of $M$ with respect to $(X,\Delta)$. Then for any 
\bdism
m>\frac{1}{1-\tau}\binom{n+2}{2}
\edism
the map induced by the linear system $\left|\left\lceil m(K_{X}+\Delta+M)\right\rceil\right|$ is birational.
\end{theorem}

In order to give a uniform bound, it suffices to bound away from one the pseudo-effective threshold. 
If $(X,\Delta;M)\in\D_{n}(I,\nn)$ then it follows from Theorem \ref{theo}. 
Note that we don't assume that $(X,\Delta+M)$ is lc, but it is close to be lc.

\begin{corollary}
Fix a positive integer $n$ and a DCC set $I\subseteq \left[0,1\right]$. Then there exists $m_{n,I}$ such that 
the map associated to the linear system $\left\lceil m(K_{X}+\Delta+M)\right\rceil$ is birational for any $m\geq m_{n,I}$
and any $(X,\Delta;M)\in\D_{n}(I,\nn)$ with $M$ a big and nef Cartier divisor.
\end{corollary}

Combining Theorem \ref{effective} and Corollary \ref{ample} we get the following.

\begin{corollary}\label{iitaka}
Fix a positive integer $n$ and a DCC set $I\subseteq \left[0,1\right]$. Then there exists a positive integer $m_{n,I}$ such that if 
\begin{enumerate}
\item $X$ is a projective variety with $\dim(X)=n$,
\item $(X,\Delta)$ is a klt pair such that $-(K_{X}+\Delta)$ is nef,
\item the coefficients of $\Delta$ are in $I$,
\item $M$ is a big and nef Cartier divisor such that $K_{X}+\Delta+M$ is big,
\end{enumerate} 
then the map associated to the linear system $\left\lceil m(K_{X}+\Delta+M)\right\rceil$ is birational for any $m\geq m_{n,I}$
\end{corollary}

Using the canonical bundle formula, Corollary \ref{iitaka} implies the uniformity of the Iitaka fibration when the image 
has nef anti-canonical divisor. See Section 4 in \cite{Gabriele} for details.

Note that if $\tau=0$ in Theorem \ref{effective} we don't even need that the coefficients of $\Delta$ are in a DCC set.
Analogously we have.

\begin{theorem}
Fix a positive integer $n$ and a positive real number $\delta$. Then there exists a positive integer $m_{n,\delta}$ such that 
if 
\begin{enumerate}
\item $(X,\Delta)$ is a lc pair,
\item $\dim(X)=n$,
\item $K_{X}$ is pseudo-effective,
\item $\Delta=\sum a_{i} \Delta_{i}$ with $a_{i}\geq \delta$,
\item $K_{X}+\Delta$ is big,
\end{enumerate}

then the map associated to $\left\lfloor m(K_{X}+\Delta)\right\rfloor$ is birational for any $m\geq m_{n,\delta}$. 
\end{theorem}

\begin{proof}
Since $K_{X}$ is pseudo-effective, $\tau(X,\Delta)=0$. This implies that $K_{X}+\alpha D$ is big for any $\alpha\geq 0$. 
Choose $\alpha:=\min\left\{1,\delta\right\}$. Then Theorem C in \cite{HMX} gives the result.
\end{proof}

If we want to deal with effective birationality results, maybe the condition $(X,\Delta+\tau M)$ is lc can be substituted with a more  suitable condition. A possible way to overcome this difficulty may be the following.

\begin{question}\label{question}
Let $(X,\Delta)$ be a klt pair and let $M$ be a big and nef Cartier divisor on $X$. Suppose that $K_{X}+\Delta+M$ is big and $\vol(K_{X}+\Delta+M)<1$. Is the pseudo-effective threshold $\tau(X,\Delta;M)$ uniformly bounded away from one?
\end{question}

Even if the previous question seems very restrictive, using the argument in \cite{Jiang} and Theorem \ref{effective}, it implies the uniformity of the Iitaka fibration when the variation is maximal. See \cite{Gabriele} for more details. \\

\noindent\textbf{Acknowledgments}.
I would like to express my gratitude to Professor J\'anos Koll\'ar for his constant support and many constructive comments. I also would like to thank Professor Christopher Hacon for generously sharing his insight with me. In particular his ideas helped me in the proof of Theorem \ref{spectrum} when $M$ is not globally generated. Part of this work was written while the author visited the University of Utah.

\end{document}